\documentclass[oneside,english,10pt]{amsart}

\usepackage[T1]{fontenc}
\usepackage[utf8]{inputenc}
\usepackage{hyperref, ulem}
\usepackage{todonotes}

\usepackage{amsmath,amssymb,amsfonts,amsthm,color,latexsym}
\usepackage{xcolor}

\usepackage{cancel}

\usepackage{color}
\usepackage{enumerate}
\usepackage{graphicx}
\usepackage{wrapfig}

\usepackage{amssymb}
\usepackage{float}

\newcommand{\co}{\mathbb{C}}
\newcommand{\cn}[1]{\mathbb{C}^{#1}}

\newcommand{\cpt}[1]{\mathbb{C}P^{2}}

\newcommand{\F}{\mathcal{F}}

\newcommand{\G}{\mathcal{G}}

\usepackage{amsmath}

\usepackage{amsmath,amssymb,amsfonts,amsthm}
\newtheorem{theorem}{Theorem}[section]

\usepackage{comment}

\newtheorem{proposition}[theorem]{Proposition}
\newtheorem{corollary}[theorem]{Corollary}
\newtheorem{lemma}[theorem]{Lemma}
\theoremstyle{definition}
\newtheorem{definition}[theorem]{Definition}

\theoremstyle{remark}
\newtheorem{remark}[theorem]{Remark}

\newcommand{\C}{\mathbb{C}}

\definecolor{verdeuva}{RGB}{122,154,1}

\begin{document}

\title[Nilpotent Holomorphic Foliations]{Nilpotent Holomorphic Foliations on \((\cn{n},{\bf 0})\) }
\date{\today}

\author{Evelia R. Garc\'{\i}a Barroso}
\address[Evelia R. Garc\'{\i}a Barroso]{Dpto. Matem\'{a}ticas, Estad\'{\i}stica e Investigaci\'on Operativa.
Instituto Universitario de Matemáticas y Aplicaciones (IMAULL). Universidad de La Laguna. Apartado de Correos 456. 38200 La Laguna, Te\-ne\-rife, Spain.}
\email{ergarcia@ull.es}

\author{Hern\'an Neciosup-Puican}
\address[Hernán Neciosup-Puican]{Dpto. Ciencias - Secci\'on Matem\'aticas, Pontificia Universidad Cat\'olica del Per\'u.}
\curraddr{Av. Universitaria 1801, San Miguel, Lima 32, Peru}
\email{hneciosup@pucp.edu.pe}


\begin{abstract}
In this paper, we study nilpotent holomorphic foliations in complex dimension $n+1$, at the origin,
 defined by germs of integrable 1-forms whose linear part is given by \(zdz\).
 These foliations generalize the classical nilpotent foliations in dimension two. We show that every nilpotent foliation in higher dimensions can be described as the pullback of Takens' normal form, which naturally leads to the existence of cuspidal hypersurfaces as invariant sets.

\noindent We focus on the case where the separatrix is a quasi-ordinary cuspidal hypersurface, and we provide a characterization of those foliations that are of generalized hypersurface type. Furthermore, we recall the Newton polyhedron of a foliation and prove that, for foliations with a quasi-ordinary cuspidal separatrix, being of generalized hypersurface type is equivalent to the coincidence of the Newton polyhedra of the foliation and its separatrix. This provides a new criterion that extends previous results studied by Fernández-S\'anchez and J. Mozo (2006), and later partially generalized by the second-named author.
\end{abstract}

\maketitle

\section{Introduction}

In dimension $(\mathbb{C}^2,\mathbf{0})$, a foliation is called \textit{nilpotent} if it is generated by a germ of an integrable holomorphic 1-form whose nonzero linear part is of the form 
\(y\,dy\). That is, foliations defined by the 1-form:
\[y\,dy+\cdots\]
where the ellipsis represents higher-order terms in $dx$ and $dy$.\\

In 1974, F. Takens \cite[page 55]{takens1974singularities} proved that such foliations are given by the normal 1-form:
\begin{equation}\label{Takens}
  \omega=(y+a(x))dy+b(x)dx,  
\end{equation}
where \(a(x)=a_px^p+a_{p+1}x^{p+1}+\cdots,\; b(x)=b_{s-1}x^{s-1}+b_sx^s+\cdots\) are some formal power series. 

On the other hand, in 2002 E.~Str\'{o}zyna and H.~\.Zo\l{}\k{a}dek \cite[Lemma 3 in Section 3]{Strozina&Zoladek}, reduce the 1-form (\ref{Takens})  to
\[\omega=d\left(\frac{1}{2}y^2-x^s\right)+a(x)dy.\]
Then, formally, every foliation with a nilpotent singularity is generated by a 1-form of the type
\begin{equation}\label{formaStrozyna&Zoladek}
    \omega_{s,p}=d(y^2-x^s)+x^pu(x)dy,
\end{equation}
where $s\geq3,\;p\geq 2,\; u\in\mathbb{C}[|x|],\;$ with $u(0)\neq 0$. In the case $s<2p$ a  normal form equivalent to (\ref{formaStrozyna&Zoladek})  is
\begin{equation}\label{formacuspidal-Loray&Strozyna&Soladek}
    d(y^2-x^s)+c(x)(2xdy+sydx),
\end{equation}
(see \cite{cuspidalFrankLoray} and \cite[Lemma 3, Section 3]{Strozina&Zoladek}). The 1-form (\ref{formacuspidal-Loray&Strozyna&Soladek}) has the property that the cusp curve $y^2-x^s=0$ is invariant. In \cite{PFJM-Cuspideal-casi-Ordinario}, the authors refer to these foliations as cuspidal foliations.

E.~Str\'{o}zyna and H.~\.Zo\l{}\k{a}dek proved in \cite[Theorem 6 in Section 2]{Strozina&Zoladek}  that the forms (\ref{Takens}, \ref{formaStrozyna&Zoladek}, \ref{formacuspidal-Loray&Strozyna&Soladek}) can be chosen analytic and they call them {\it Takens prenormal forms}. The case $s=3$ was studied by R. Moussu \cite{Moussu-Holonomie} and a generalizations when $s\geq 3$ and $2p>s$ by Cerveau and Moussu \cite{cerveau1988groupes}. In both cases, the reductions of the singularities agree with the desingularization of the curve $y^2-x^s=0$.\\

Put $\mathcal{O}_n:=\mathbb C\{{\bf x}\}=\mathbb C\{x_1,x_2,\ldots,x_n\}$ for all $n\in \mathbb N$, $n>0$.\\

D. Cerveau and R. Moussu (\cite{cerveau1988groupes}, page 478, statement 1) proved, in the case $2p>s$, there exist analytic coordinates $(x,y)$ on $(\mathbb{C}^2,{\bf 0})$ in which (\ref{formaStrozyna&Zoladek}) can be written, up to multiplication by a unit, as:
\[\eta=d(y^2-x^s)+g(x,y)(sydx+2xdy),\]
where $g\in\mathcal{O}_2$ with $g(\mathbf{0})=0$. In this case, the foliation defined by $\eta$ admits the curve $y^2-x^s=0$ as its set of separatrices (\cite[Lemma 5.2]{Ovidius}).\\

\noindent
In this work, we extend the study of these foliations to the case $(\mathbb{C}^3,\mathbf{0})$, focusing on those of {\it cuspidal type} associated with holomorphic codimension-one foliations {\it of generalized surface type}. We show that such foliations can be effectively characterized by means of the notion of {\it weighted order} and the geometry of the {\it Newton polyhedron} of their defining 1-form. This approach provides a solid algebraic and combinatorial framework for the classification of a family of generalized surface-type foliations in dimension three. Furthermore, our results are obtained {\it without resorting to the resolution of singularities}, which yields an alternative, direct, and efficient method for the local study of complex foliations with cuspidal behavior.

\section{nilpotent holomorphic foliations on $(\mathbb{C}^{n+1}, {\bf 0})$}

J.P. Dufour and M. Zhitomirskii established in \cite{Dufour}  a formal normal form for singular integrable 1-forms whose linear part is non-zero. Later, F. Loray \cite{loray2006preparation} proved the convergence of this normal form. More precisely:

Let $\mathcal{F}$ be a singular holomorphic foliation of codimension one on $(\mathbb{C}^{n+1}, {\bf 0})$, defined by an integrable 1-form $\omega\in\Omega^1(\mathbb{C}^{n+1},{\bf 0})$
\begin{equation}\label{forma-zk}
\omega=z^kdz+B({\bf x},z)dz+\sum_{i=1}^nA_i({\bf x},z)dx_i,
\end{equation}
where $A_i, B\in \mathcal{O}_{n+1}$ are germs of holomorphic functions, with  $A_i({\bf 0},0)=B({\bf 0},0)=0$ and $k\in\mathbb{N}$.

If $B\equiv 0$, J.P. Dufour and M. Zhitomirskii (\cite[Theorem 3.7]{Dufour}) proved that, up to multiplication by a unit, $\omega$ is formally equivalent to
\begin{equation}\label{forma-formal-Dunfour}
\omega=z^kdz+\sum_{i=0}^kz^i\omega_i,
\end{equation}
where $\omega_i$ are 1-forms depending only on $x_1,x_2,\ldots,x_n$.

On the other hand, F. Loray proves that the 1-form (\ref{forma-zk}) is analytically equivalent to 
\begin{equation}\label{forma-Loray}
\omega=\sum_{i=1}^nP_i({\bf x},z)dx_i+Q({\bf x},z)dz,
\end{equation}
where $P_i, Q\in\mathcal{O}_{n}[z]$, with degree $\leq k$  and $Q$ monic.
If $\mathcal{F}$ es nilpotent, that is, $k=1$, then 
\[\begin{array}{rcl}
P_i({\bf x},z)  &=&a_i({\bf x})z+b_i({\bf x}),\\
Q({\bf x},z)    &=&z+g({\bf x}),
\end{array}\]
where $a_i, b_i, g\in\mathcal{O}_n$ with $a_i({\bf 0})=b_i({\bf 0})=g({\bf 0})=0$.
Performing the linear change in the coordinate $z$, $z\longmapsto z-g({\bf x})$, we obtain
\begin{eqnarray*}
\omega&=&\displaystyle\sum_{i=1}^{n}(a_{i}({\bf x})-g({\bf x})b_{i}({\bf x}))dx_i+z\sum_{i=1}^{n}b_{i}({\bf x})dx_{i}+zdz,\\
&=&\omega_0+z\omega_{1}+zdz 
\end{eqnarray*}

\noindent where \(\omega_{i}\in\Omega^1(\mathbb{C}^n,{\bf 0})\) are germs of 1-forms depending only on the variables \(x_1,x_2,\ldots,x_n\), with multiplicity \({\rm mult}_0(\omega_i)\geq 2.\) 
The integrability condition \(\omega\wedge d\omega=0\) implies that \(\omega_{i}=df_{i},\) where \(f_{i}\in\mathcal{O}_n\). Therefore, we can write: 
\[\omega=df_0+zdf_1+zdz,\]

\noindent where \({\rm mult}_0(f_i)>2\), and the integrability condition translates into \(df_0\wedge df_1=0\). Following \cite{hneciosup1}, a further change of variable \(z\longmapsto z-f_1({\bf x})\) allows us to rewrite the 1-form \(\omega\) as \[\omega=df_0-f_1dz+zdz.\] Moreover, the condition \(df_0\wedge df_1=0\) implies that the functions \(f_{i}\) are functionally dependent. R. Moussu (\cite[Chapter II]{Moussu1976}) proved that there exist a primitive function germ \(f\in\mathcal{O}_n\) and analytic functions \(h_{0}, h_{1}\in\co\{t\}\) such that \[f_{i}=h_{i}\circ f.\] Thus, considering the map \(\varphi: (\cn{n+1},{\bf 0}) \to (\cn{2},{\bf 0})\) defined by: \[\varphi({\bf x},z)=(t,z)=(f({\bf x}),z),\] we obtain \(\omega=\varphi^*\omega_{0}\), where 
\begin{eqnarray*}
   \omega_{0}&=&zdz+h'_{0}(t)dt-h_{1}(t)dz\\
&=&(z+a(t))dz+b(t)dt, 
\end{eqnarray*}

\noindent with 
${\rm mult}_0(a)=p\geq 3$, and ${\rm mult}_0(b)=s-1\geq2$. Note that the form $\omega_0$ is the Takens form (\ref{Takens}). Then, according to \cite{Strozina&Zoladek}, after an analytic change in the coordinate $t$, $\omega_0$  can be written as
\[\omega_{s,p}=d(z^2-t^s)+t^pu(t)dz.\]

As a result of the preceding discussion, we obtain

\begin{lemma}\label{pull-back-nilpotente} Every nilpotent type foliation \(\mathcal{F}\) on \((\cn{n+1},{\bf 0})\) is the pullback of Takens' formal normal form. 
\end{lemma}

From Lemma \ref{pull-back-nilpotente}, up to multiplication by a scalar, every nilpotent foliation on $(\mathbb{C}^{n+1},{\bf 0})$ is generated by
\[\omega=d(z^2-f^s({\bf x}))+f^{p}({\bf x})u(f)dz,\]
where $s\geq 3,\,p\geq 2$ and $u\in\mathcal{O}_n$ with $u({\bf 0})\neq0$.

In the case where $2p > s$, following \cite{cerveau1988groupes}, there exist analytic coordinates $(x, y)$ centered at the origin of $(\mathbb{C}^2, \mathbf{0})$ in which the form $\omega_{s,p}$ can be written, up to multiplication by a unit, as:
\[
\eta_{s,2} = d(y^2 - x^s) + g(x, y)(s y\, dx + 2x\, dy),
\]
where $g \in \mathcal{O}_2$ satisfies $g(\mathbf{0}) = 0$. In this setting, the foliation defined by $\eta_{s,p}$ admits the curve $y^2-x^s=0$ as its set of separatrices.

Before stating the next result, we recall the notion of a {\it cuspidal hypersurface}, which arises naturally in the study of nilpotent foliations with singular separatrix sets.

\begin{definition}\label{cuspidal_hypersurface}
Let \( S \subset (\mathbb{C}^{n+1}, \mathbf{0}) \) be a germ of an \( n \)-dimensional hypersurface with coordinates \( (x_1, \ldots, x_n, z) \). \( S \) is a \textit{cuspidal hypersurface} if it is defined by an analytic equation of the form
\[
z^2 + g(x_1, \ldots, x_n) = 0,
\]
with \( \mathrm{mult}_0(g) \geq 2 \).
\end{definition}
This class of hipersurfaces include, in particular, the so-called {\it quasi-ordinay cuspidal hypersurfaces}, defined by equations of the form
\[z^2+{\bf x}^P=0,\quad\text{where}\quad {\bf x}^P=x_1^{p_1}x_2^{p_2}\cdots x_n^{p_n}.\]

Cuspidal hypersurfaces naturally arise as a special case of the so-called {\it Zariski surfaces}, introduced by O. Zariski and named by Piotr Blass in 1980. 
These surfaces are defined by equations of the form
\[
z^k + g(x_1, x_2) = 0, \quad k \geq 2,
\]
where \( g \in \mathcal{O}_2 \) with \( \mathrm{mult}_0(g) \geq 2 \). They were extensively studied in the 1980s by Joseph Blass, Piotr Blass, and Jeff Lang \cite{Blas&Priot}, \cite{Blass-Priot-ZariskiSurface},  \cite{Blas-Priot&Levine}. For the special case \( k = 2 \), the surface becomes cuspidal according to Definition \ref{cuspidal_hypersurface}. Related works include those of A. Pichon \cite{Pichon&Anne}, and R. Mendris and A. Némethi \cite{Mendris&Nemethi}, dealing with irreducible plane curve singularities.

The following lemma is a direct consequence of \cite[Theorem 1.4]{hneciosup1}:

\begin{lemma}
Let \( \mathcal{F} \) be a nilpotent foliation on \((\mathbb{C}^{n+1},{\bf 0})\) admitting a cuspidal hypersurface \( S : z^2 + f^s(x) = 0 \) as its separatrix set. Then, there exist analytic coordinates in which \( \mathcal{F} \) is defined by the integrable 1-form
\begin{equation}\label{omega_s2}
\omega_{s,2} = d(z^2 - f^s({\bf x})) +  g(f({\bf x}), z) \left(s\,z\,df({\bf x}) + 2\, f({\bf x})dz\right),
\end{equation}
where \( g \in \mathcal{O}_2 \).
\end{lemma}

We denote by \( \Sigma_{s,2} \) the set of foliations defined by a 1-form $\omega$ that is analytically equivalent to \( \omega_{s,2} \). Let \( \mathcal{F} \in \Sigma_{s,2} \). Since \( s \geq 3 \), the singular locus of \( \mathcal{F} \) is given by
\[
\mathrm{Sing}(\mathcal{F}) = \{ (\mathbf{x}, 0) \in \mathbb{C}^{n+1} : f(\mathbf{x}) = 0 \}.
\]

Expanding the function $g$ from the expression of (\ref{omega_s2}) in its Taylor series, we write:
\[g(u,v)=\sum_{i+j\geq 1}c_{ij}u^iv^j.\]
We define the {\it weighted order} of $g$, denoted by ${\rm ord}_{s,2}(g)$, as:
\begin{equation}\label{ordenpesado}
    {\rm ord}_{s,2}(g)=\min\{2i+sj: c_{ij}\neq 0\}.
\end{equation}

\section{Generalized hypersurface}
A germ of foliation at \((\mathbb{C}^2, 0)\) is said to be of generalized curve type if there are no saddle-nodes in its reduction of singularities. In this family, the separatrices — that is, formal invariant curves — are all analytic and capture significant topological information about the foliation. For instance, a foliation of generalized curve type and its set of separatrices share the {\it same reduction of singularities} ( see \cite[Theorem 2 and Theorem 6]{CCamachoLNetoPSad}), in the sense that a sequence of blow-ups that desingularizes all separatrices transforms the foliation into one with only simple singularities.

Another characterization arises through the Milnor number. Foliations of generalized curve type minimize the Milnor number, once a set of separatrices is fixed. For example, if the set of separatrices of a foliation \(\mathcal{G}\) is finite (the so-called non-dicritical case) and \(f(x,y)=0\) is a reduced equation of the separatrix set, then \(\mathcal{G}\) is of generalized curve type if and only if \(\mu(\mathcal{G}) = \mu(df)\), where \(\mu\) denotes the Milnor number at the origin of the foliation given by $df$, that is, the intersection multiplicity of $\frac{\partial f}{\partial x}=0$ and $\frac{\partial f}{\partial y}=0$ at the origin.

A natural extension of this notion leads to the study of generalized
hypersurface foliations in ambient dimension three or higher. In dimension three, the existence of a reduction of singularities for codimension one foliations, as established in \cite{cano1993reduction}, enabled the definition of generalized surface foliations in \cite{PF-JM-GeneralizedSurfaces}. These are foliations that are non-dicritical, that is, the exceptional divisor obtained through the reduction process is invariant, and whose reduction of singularities contains no saddle-nodes.

The main theorem of \cite{PF-JM-GeneralizedSurfaces} states that a generalized surface foliation and its set of separatrices share the same {\it reduction of singularities}, thus extending to the three-dimensional setting a characteristic property of generalized curve foliations in \((\mathbb{C}^2, 0)\). However, this correspondence fails in the dicritical case, since there exist foliations — such as the well-known example of Jouanolou \cite{Jouanolou} — that admit no separatrix. This contrast underscores the importance of restricting to the non-dicritical case when attempting to extend to higher dimensions the structural properties of generalized curve foliations.\\

Recall the following definition from \cite{PF-JM-HN-GeneralizedHypersurfaces}
\begin{definition}
Let \(\mathcal{F}\) be a codimension one holomorphic foliation defined by an integrable 1-form \(\omega\) on \((\mathbb{C}^{n+1}, 0)\). We say that \(\mathcal{F}\) is of generalized hypersurface type if the following conditions are satisfied:
\begin{enumerate}
    \item \(\mathcal{F}\) is non-dicritical.
    \item For every map \(\varphi : (\mathbb{C}^2, 0) \to (\mathbb{C}^{n+1}, 0)\), generically transversal to \(\mathcal{F}\) (i.e. $\varphi^*\omega\not\equiv 0$), the pullback foliation \(\varphi^*\mathcal{F}\) is of generalized curve type.
\end{enumerate}
\end{definition}

In \cite{PF-JM-HN-GeneralizedHypersurfaces}, the main result of \cite{PF-JM-GeneralizedSurfaces} is extended to arbitrary dimension. More precisely, the authors prove the following:

\begin{theorem}[{\cite[Theorem 1]{PF-JM-HN-GeneralizedHypersurfaces}}]
Let \(\mathcal{F}\) be a foliation of generalized hypersurface type on \((\mathbb{C}^{n+1}, 0)\), and let \(S := \{f = 0\}\) be a reduced equation defining its set of separatrices. Let \(\pi : (M, E) \to (\mathbb{C}^{n+1}, 0)\) be the morphism of an embedded desingularization of \(S\). Then the pullback foliation \(\pi^*\mathcal{F}\) has only simple singularities adapted to the exceptional divisor \(E\).
\end{theorem}

This result provides substantial evidence that the fundamental relationship between a foliation and its separatrices — a characteristic property of the generalized curve case — persists in a more general class of non-dicritical foliations in higher dimensions.

In the context of generalized hypersurface foliations, it is natural to investigate to what extent specific geometric conditions on the separatrix constrain the structure of the foliation. In particular, when the separatrix is assumed to be a quasi-ordinary cuspidal hypersurface, it is possible to obtain additional information in terms of the weighted order, as shown in the following proposition:

\begin{proposition}\label{propo-casi-ordinara}
Let $\mathcal{F}\in \Sigma_{s,2}$ be a non-dicritical holomorphic foliation on $(\mathbb{C}^{n+1},{\bf 0})$ with separatrix $S$, a quasi-ordinary cuspidal hypersurface. If $\mathcal{F}$ is of generalized hypersurface type, then the weighted order of $g$ satisfies:
\[{\rm ord}_{s,2}(g)\geq s-2.\]
\end{proposition}
\begin{proof}
    We begin with the expression
    \[\omega_{s,2}=d(z^2-{\bf x}^{Ps})+{\bf x}^P\,z\,g({\bf x}^P, z)\left(2{dz\over z}+s\sum_{i=1}^np_i{dx_i\over x_i}\right).\]
    By hypothesis, $\mathcal{F}$ is of generalized hypersurface type. Then, for any generically transverse analytic section
    \[\varphi:(\mathbb{C}^2,{\bf 0})\to(\mathbb{C}^{n+1},{\bf 0}),\]
    the pullback foliation $\mathcal{G}=\varphi^*\mathcal{F}$ is of generalized curve type. Consider \[\varphi(u,v)=(c_0u,c_1u,c_2u,\cdots,c_{n-1}u,v),\quad\text{with}\quad c_i\in\mathbb{C}^*.\] Without loss of generality, we may assume $c_i=1$ for all $i$. 
    Then the foliation $\mathcal{G}$ is defined by:
    \[\eta=\varphi^*\omega_{s,2}=d(v^2-u^q)+\Delta(u,v)(2udv+qvdu),\]
    where $q=s|P|$ with $|P|=p_1+\cdots +p_n$, and \[\Delta(u,v)=u^{|P|-1}g(u^{|P|},v).\] Following \cite[Proposition 2.4.9]{Thesis-Saravia} we have
    \begin{equation}\label{ordenpesado2}
        {\rm ord}_{q,2}\Delta\geq {q-2}.
    \end{equation}

    On the other hand, writing the Taylor expansion of $g$, we obtain
    \[\Delta(u,v)=\sum_{i+j\geq 1}c_{ij}u^{|P|(i+1)-1}v^j.\]
  
    \noindent Thus, according to (\ref{ordenpesado}), we have
    \[
\begin{array}{rcl}
   {\rm ord}_{q,2}(\Delta) &=&\min\left\{ 2\left(|P|(i+1)-1\right)+q\cdot j:c_{ij}\neq0\right\}\\
                            &=&\min\left\{ 2(|P|(i+1)-1)+s|P|\cdot j:c_{ij}\neq0\right\}\\
                            &=&\min\left\{ |P|(2i + sj)+2|P|-2:c_{ij}\neq0\right\}\\
                            &=& |P| \cdot {\rm ord}_{s,2}(g) + 2|P| - 2. 
    \end{array}
    \]
    Consequently, by (\ref{ordenpesado2}),
   
   \[|P|\cdot{\rm ord}_{s,2}(g)+2|P|-2\geq q-2.\]
    Simplifyng, we obtain
\[{\rm ord}_{s,2}(g)\geq s-2. \]
\end{proof}

In the case $n=2$, that is, on $(\mathbb{C}^3,{\bf 0})$, after applying \cite[Theorem 5.2]{hneciosup1}, the condition given in Proposition \ref{propo-casi-ordinara} also becomes sufficient:

\begin{corollary}\label{cor:cond}
Let $\mathcal{F}\in \Sigma_{s,2}$ be a non-dicritical foliations on $(\mathbb{C}^3,{\bf 0})$ with quasi-ordinary cuspidal separatrix. Then, $\mathcal{F}$ is of generalized surface type if only if
\[{\rm ord}_{s,2}(g)\geq s-2.\]
\end{corollary}

The condition on the order given in Corollary \ref{cor:cond} for
$n=2$, is a sufficient condition in \cite[Theorem 5.2]{hneciosup1}, which is obtained by studying the reduction of singularities of the foliation; however, this approach is not currently possible in higher dimensions.

\section{Newton Polyhedron}

The Newton polyhedron of any nonempty subset $S \subset \mathbb{R}_+^{n+1}$, denoted by $\mathcal{N}(S)$,  is the boundary of the convex hull of $S + \mathbb{R}_+^{n+1}$, where $+$ denotes de Minskowski sum.

In this section we will recall the notion of the Newton's polyhedron of foliations.

Let \( \omega \) be a holomorphic 1-form on \( (\mathbb{C}^{n+1}, \mathbf{0}) \). In coordinates \( ({\bf x},z):=(x_1, \ldots, x_n,z) \), it can be written as
\[
\omega = \sum_{i=1}^n a_i({\bf x},z) \, dx_i+a_{n+1}({\bf x},z)dz,
\]
or equivalently,
\begin{equation}\label{forma}
\omega = \sum_{i=1}^n x_i a_i(\mathbf{x},z) \frac{dx_i}{x_i}+za_{n+1}({\bf x},z){dz\over z},
\end{equation}
where \( a_i(\mathbf{x},z) \in \mathcal{O}_{n} \).

\noindent Expanding each \( a_i(\mathbf{x},z) \) into its Taylor series, we have
\[
a_i(\mathbf{x},z) = \sum_{|\alpha| \geq 0} a_\alpha^i (\mathbf{x},z)^\alpha, \quad a_\alpha^i \in \mathbb{C}.
\]
Thus, we can express
\[
\omega = \sum_{i=1}^n x_i \sum_{|\alpha| \geq 0} a_\alpha^i (\mathbf{x},z)^\alpha \frac{dx_i}{x_i}+za_{\alpha}^{n+1}({\bf x},z)^{\alpha}{dz\over z},
\]
and regrouping terms,
\[
\omega = \sum_{\alpha \in A} (\mathbf{x},z)^\alpha \omega_\alpha,
\]
where \( A \subset \mathbb{Z}^{n+1}_+ \subset \mathbb R_+^{n+1}\) and \( \omega_\alpha = \sum_{i=1}^n a_\alpha^i \frac{dx_i}{x_i}+a_{\alpha}^{n+1}{dz\over z} \).

\begin{definition}
Let \( \omega \) be a holomorphic 1-form on \( (\mathbb{C}^{n+1}, \mathbf{0}) \). The \textit{Newton polyhedron} of \( \omega \), denoted by \( \mathcal{N}(\omega) \), is \( \mathcal{N}(A) \).
\end{definition}

Alternatively, if a 1-form is written as in \eqref{forma}, we define the \textit{support} of \( \omega \) as
\[
\mathrm{supp}(\omega) = \left(\bigcup_{i=0}^n \mathrm{supp}(x_i a_i(\mathbf{x},z))\right)\cup\mathrm{supp}(za_{n+1}({\bf x},z)),
\]
where \( \mathrm{supp}(x_i a_i(\mathbf{x}),z)\) denotes the support of the monomials appearing in \( x_i a_i(\mathbf{x}) \).  
Accordingly, the Newton polyhedron of $\omega$ coincides with $\mathcal{N}(\mathrm{supp}(\omega))$.

\begin{remark}
If the holomorphic 1-form \( \omega \) is the differential of a holomorphic function $F$, i.e., \( \omega = dF \), then
\[
\mathcal{N}(\omega) = \mathcal{N}(F),
\]
where $\mathcal{N}(F)$ denotes the Newton polyhedron of the support of $F$.
\end{remark}

Observe that if $\varphi \in  \C\{x_1,\ldots,x_{n+1}\}$ is a unit and
$\omega' = \varphi \omega$,  then $\mathcal{N}(\omega') = \mathcal{N}({\omega})$ and, hence,
 the Newton polyhedron is an invariant of the foliation $\F$ defined by $\omega$ (respectively $\omega'$) and it will be denoted by $\mathcal{N}(\mathcal{F})$.

\begin{theorem}\label{thm-PolNewton}
Let \( \mathcal{F} \in \Sigma_{s,2} \) be a foliation on \( (\mathbb{C}^{n+1}, \mathbf{0}) \) with quasi-ordinary cuspidal separatrix. If \( \mathcal{F} \) is of generalized hypersurface type, then there is a change of coordinates under which the Newton polyhedra of \( \mathcal{F} \) and its separatrix \( S: (F({\bf x},z)=0) \) coincide, that is,
\[
\mathcal{N}(\mathcal{F}) = \mathcal{N}(F).
\]
\end{theorem}

\begin{proof}
    Suppose that \( \mathcal{F}\), defined by the 1-form
\[\omega_{s,2}=d(z^2-{\bf x}^{Ps})+{\bf x}^P\,z\,g({\bf x}^P, z)\left(2{dz\over z}+s\sum_{i=1}^np_i{dx_i\over x_i}\right),\]
is of generalized hypersurface type. By definition, for any generically transverse section
\[
\phi: (\mathbb{C}^2, \mathbf{0}) \to (\mathbb{C}^{n+1}, \mathbf{0}),
\]
the pullback foliation \( \mathcal{G} := \phi^*\mathcal{F} \) is of generalized curve type. \\
Consider $P:=(p_1,p_2,\ldots,p_n)\in \mathbb N^n$ and the linear map
\[\phi(u,v)=\left({p_1u\over P^{P\over |P|}}, {p_2u\over P^{P\over |P|}}, \ldots, {p_nu\over P^{P\over |P|}},v\right),\]
where $|P|=p_1+p_2+\cdots+p_n$ and $P^{P\over |P|}=p_1^{p_1\over |P|}\cdot p_2^{p_2\over |P|}\cdots p_n^{p_n\over |P|}$.

The foliation $\G$ is defined by

\[\phi^*\omega_{s,2}=d(v^2-u^{s|P|})+u^{|P|}vg(u^{|P|},v)\left(2{dv\over v}+s|P|{du\over u}\right).\]

Then, by \cite[Theorem 1.1]{Ovidius}, there is a coordinate system such that 
\[
\mathcal{N}(\mathcal{G}) = \mathcal{N}(C),
\]
where \(C:v^2-u^{s|P|}=0\) is the separatrix of \(\mathcal{G}\).

Observe that
\[
\phi(\mathrm{supp}(\eta)) = \mathrm{supp}(\omega_{s,2}), \quad \text{and} \quad \phi(\mathrm{supp}(C)) = \mathrm{supp}(F).
\]
Moreover,
\[
\phi\left( \mathrm{conv}\left( \mathrm{supp}(\eta) + \mathbb{R}^{n+1}_+ \right) \right) = \phi\left( \mathrm{conv}\left( \mathrm{supp}(C) + \mathbb{R}^{n+1}_+ \right) \right),
\]

\noindent and since \( \phi \) is linear,
 
\[
\mathrm{conv}\left( \phi(\mathrm{supp}(\eta)) + \phi(\mathbb{R}^{n+1}_+) \right) = \mathrm{conv}\left( \phi(\mathrm{supp}(C)) + \phi(\mathbb{R}^{n+1}_+) \right),
\]
\noindent thus,
\[
\mathrm{conv}\left( \mathrm{supp}(\omega_{s,2}) + \phi(\mathbb{R}^{n+1}_+) \right) = \mathrm{conv}\left( \mathrm{supp}(F) + \phi(\mathbb{R}^{n+1}_+) \right).
\]
This implies that
\[
\mathcal{N}(\mathcal{F}) = \mathcal{N}(F).
\]
\end{proof}

In the case $n=2$, that is, on $(\mathbb{C}^3,{\bf 0})$, after applying Corollary \ref{cor:cond}, the condition given in Theorem \ref{thm-PolNewton} also becomes sufficient:

\begin{corollary}
Let \( \mathcal{F} \in \Sigma_{s,2} \) be a foliation on \( (\mathbb{C}^{3}, \mathbf{0}) \) with quasi-ordinary cuspidal separatrix. Then \( \mathcal{F} \) is of generalized surface type if and only if there is a change of coordinates under which the Newton polyhedra of \( \mathcal{F} \) and its separatrix \( S: (F({\bf x},z)=0) \) coincide, that is,
\begin{equation}\label{Newton}
\mathcal{N}(\mathcal{F}) = \mathcal{N}(F).    
\end{equation}
\end{corollary}

\begin{proof}
Let us first recall that, by Theorem \ref{thm-PolNewton}, condition \eqref{Newton} is necessary in any dimension. Therefore, to prove the corollary, it suffices to show that in dimension $3$ condition \eqref{Newton} is also sufficient.

The foliation \( \mathcal{F} \in\Sigma_{s,2}\) is  defined by
\[
\omega_{s,2} = d(z^2 - x^{ps}y^{qs}) + g(x^{p}y^{q}, z) x^{p}y^{q}z\left( 2 \frac{dz}{z} + sp \frac{dx}{x} + sq \frac{dy}{y} \right),
\]
and, upon expanding the differential, we can write
\[
\omega_{s,2} = a(x,y,z) \left( p \frac{dx}{x} + q \frac{dy}{y} \right) + 2b(x,y,z) \frac{dz}{z},
\]
where
\[
\begin{array}{rcl}
a &=& s(x^{p}y^{q})^{s} - s g(x^{p}y^{q},z) x^{p}y^{q}z, \\
b &=& z^2 + g(x^{p}y^{q},z) x^{p}y^{q}z.
\end{array}
\]
Expanding \( g \) into its Taylor series,
\[
g(x^{p}y^{q},z) = \sum_{i+j>0} c_{ij} x^{ip} y^{iq} z^j,
\]
we obtain
\[
\mathrm{supp}(\omega_{s,2}) = \mathrm{supp}(a,b) = \left\{ (0,0,2), (sp,sq,0), ((i+1)p, (i+1)q, j+1) : c_{ij} \neq 0 \right\}.
\]

On the other hand, the support of the separatrix, defined by
\[
F(x,y,z) = z^2 + (x^p y^q)^s,
\]
is given by
\[
\mathrm{supp}(F) = \left\{ (0,0,2), (sp,sq,0) \right\}.
\]

Since by hypothesis \( \mathcal{N}(\mathcal{F}) = \mathcal{N}(F) \), it follows that
\begin{equation}\label{condiciones}
i+1 \geq s, \quad \text{or} \quad i+1 < s \; \text{and} \; j \geq 1, \quad \text{or} \quad i+1 < s \ \text{and} \ j = 0.
\end{equation}
Moreover, \( \mathcal{N}(\mathcal{F}) \) coincides with the Newton polyhedron  of the set $\{(0,0,2), (sp, sq,0)\}$.

After Corollary \ref{cor:cond} suppose, by contradiction, that \( \mathrm{ord}_{s,2}(g) < s-2 \). Then there exist \( i_0, j_0 \in \mathbb{Z}_{\geq 0} \) such that
\[
2i_0 + sj_0 < s-2,
\]
which implies that
\[
2(i_0 + 1) + sj_0 < s,
\]
and therefore,
\[
i_0 + 1 < \frac{s}{2}(1-j_0),
\]
contradicting the conditions in \eqref{condiciones}.

\end{proof}
\bigskip

\noindent {\bf Funding:} The first-named  author was partially supported by the  Spanish grant  PID2023-149508NB-I00 funded by 
MICIU/AEI
/10.13039/501100011033 and by
FEDER, UE. The second-named author was partially supported by the PUCP Postdoctoral Fellowships 2024.\\

\noindent {\bf Acknowledgements:}
We are grateful to Arturo Fernández-Pérez for his remarks.

\end{document}